\documentclass[12pt]{amsart}
 \usepackage{amssymb, amsthm, amsmath, mathrsfs}
\usepackage[all,cmtip]{xy}
\usepackage{xcolor}

\DeclareSymbolFont{largesymbols}{OMX}{yhex}{m}{n}
\DeclareMathAccent{\widehat}{\mathord}{largesymbols}{"62}

\newtheorem{theorem}{Theorem}[section]
\newtheorem{corollary}[theorem]{Corollary}
\newtheorem{lemma}[theorem]{Lemma}
\newtheorem{proposition}[theorem]{Proposition}
\newtheorem{question}[theorem]{Question}

\theoremstyle{definition}
\newtheorem{definition}[theorem]{Definition}
\newtheorem{remark}[theorem]{Remark}

\DeclareMathOperator{\Ima}{im}

\newcommand{\End}{{\rm End}}
\newcommand{\hX}{ \widehat{X}}
\newcommand{\tX}{\widetilde{X}}
\newcommand{\hphi}{\widehat{\varphi}}
\newcommand{\tphi}{\widetilde{\varphi}}
\newcommand{\raA}{\overrightarrow{A}}

\newcommand{\rk}{{\rm rk}}
\newcommand{\mrk}{{\rm mrk}}
\newcommand{\diam}{{\rm diam}}

\newcommand{\mdim}{{\rm mdim}}

\newcommand{\Wdim}{{\rm Wdim}}

  \newcommand{\cE}{{\mathcal E}}

 \newcommand{\bN}{{\mathbb N}}
 
 \newcommand{\bR}{{\mathbb R}}
 \newcommand{\bT}{{\mathbb T}}
 \newcommand{\bZ}{{\mathbb Z}}

\begin{document}

\title{Mean dimension of natural extension of algebraic systems}

\author{Bingbing Liang}
\author{Ruxi Shi}

\address{\hskip-\parindent
B.L., Department of Mathematical Science, Soochow University,
Suzhou 215006, China.}
\email{bbliang@suda.edu.cn}

\address{\hskip-\parindent
R.S., Sorbonne Universite, LPSM, 75005 Paris, France}
\email{ruxi.shi@upmc.fr}


\begin{abstract}
Mean dimension may decrease after taking the natural extension. In this paper we show that mean dimension is preserved by natural extension for an endomorphism on a compact metrizable  abelian group. As an application, we obtain that the mean dimension of an algebraic cellular automaton coincides with the mean dimension of its natural extension, which strengthens a result of Burguet and Shi \cite{BS21} with a different proof.
\end{abstract}

\maketitle

\section{Introduction}

By a dynamical system, we mean a pair $(X, \varphi)$ consisting of a compact metrizable space $X$ and a continuous map $\varphi \colon X \to X$. The natural extension of $(X, \varphi)$, denoted by $(\tX,\tphi)$ (see Definition \ref{natural extension}), preserves many important dynamical properties of $(X, \varphi)$.  It is well known that the topological entropy of $(X, \varphi)$ coincides with the topological entropy of $(\tX, \tphi)$. In this sense we say that topological entropy is preserved by natural extension.

Mean (topological) dimension is a dynamical invariant which can distinguish dynamical systems with infinite entropy. Gromov introduced this invariant when measuring the size of some holomorphic function spaces \cite{Gromov99T}. Mean dimension plays an important role in dynamical embedding problems  \cite {LW00, Gutman15, GLT16} and has close relations with other areas like information theory \cite{LT18} and operator algebras \cite{LL18, EN17}.  As pointed out by Burguet and Shi, the mean dimension of the natural extension $(\tX, \tphi)$ is no greater than the mean dimension of $(X, \varphi)$ and the inequality can be strict  \cite[Section 3.2]{BS21}. In fact, it is mainly due to a topological obstruction that the inverse limit of compact metrizable spaces can decrease the (covering) dimension \cite[Example 183]{IM12}.  This motivates the following question:
\begin{question} \label{main concern}
For which dynamical systems $(X, \varphi)$ is it true that  $\mdim(X, \varphi)=\mdim(\tX, \tphi)$?
\end{question}

In this paper, we answer Question \ref{main concern} when $(X, \varphi)$ admits an algebraic structure as in \cite{S95B}. That is,

\begin{theorem}\label{main thm}
For a compact metrizable abelian group $X$ and a continuous endomorphism $\varphi \colon X \to X$, we have $\mdim(X, \varphi)=\mdim(\tX, \tphi)$.
\end{theorem}

To prove this theorem,  taking advantage of the abelian group structure of $X$, we can instead consider the setting of an action on the Pontryagin dual space.  By introducing a proper notion of the mean rank for a general endomorphism on a discrete abelian group  (also called the rank-entropy in \cite{SZ09}), we shall prove that the mean dimension of $(X, \varphi)$ coincides with the mean rank of its Pontryagin dual. With the help of a result of Gutman \cite{Gutman17} regarding the mean dimension of non-wandering subsystems, we can finally finish the proof from the surjective case to the general case.

It is worth-mentioning that the approach of interplay between an action on a compact abelian group and its dual action on Pontryagin dual space  already appears in \cite{LL18}, where the addition formula for mean dimension of algebraic actions is hard to prove directly. However,  when transforming the problem into the algebraic setting, the discreteness of algebraic objects makes the proof accessible.

On the other hand, Burguet and Shi showed that the mean dimension of a unit cellular automaton (see Definition \ref{CA}) is preserved by natural extension \cite[Theorem 6.4]{BS21}. In general,  one may raise the following
\begin{question}\cite[Question 3.6]{BS21} \label{coincide question}
Does the mean dimension of a general cellular automaton coincide with the mean dimension of its natural extension?
\end{question}

Based on a result concerning cellular automata \cite[Lemma 7.3]{BS21}, it is implicitly proved that the mean dimension of an algebraic permutative cellular automaton  (see Definition \ref{CA}) is preserved by natural extension. The following corollary improves on this result by a different method and gives a partial answer to Question \ref{coincide question}.

\begin{corollary} \label{CA application}
For an algebraic cellular automaton $F \colon X^\bZ \to X^\bZ$, we have $\mdim(X^\bZ, F)=\mdim(\widetilde{X^\bZ}, \widetilde{F})$.
\end{corollary}

The paper is organized as follows. In Section 2 we recall the definitions of mean dimension and of natural extension of a dynamical system. In Section 3 we introduce the notion of mean rank for a general group endomorphism, discuss some properties, and establish an equality relating mean dimension and mean rank. In Section 4 we complete the proof of Theorem \ref{main thm} and Corollary \ref{CA application}.

\section{background}

\subsection{Mean dimension}
 Let $X$ be a compact metrizable space with a compatible metric $d$. Fix $\varepsilon > 0$. Recall that a continuous map $f \colon (X, d) \to Y$ into another topological space $Y$ is an {\it $(\varepsilon, d)$-embedding} if $$\diam (f^{-1}(y), d) < \varepsilon$$
for every $y \in Y$, where $\diam (f^{-1}(y), d)$ is the diameter of the set $f^{-1}(y)$ under the metric $d$.  Denote by $\Wdim_\varepsilon(X, d)$ the minimal dimension of a polyhedron $P$ such that there exists an $(\varepsilon, d)$-embedding $f \colon (X, d) \to P$.

The following important lemma was proved first by Gromov \cite[Section 1.11 and 1.12]{Gromov99T} and it can also be deduced by Brouwer's fixed point theorem  \cite[Proposition 4.6.5]{C15}.

\begin{lemma} \label{Wdim base}
For every $0 < \varepsilon < a$ and $n \in \bN$,  we have
$$\Wdim_\varepsilon ([0,a]^n, |\cdot|_\infty)=n.$$
\end{lemma}

Let $\varphi \colon X \to X$ be a continuous map.  For a compatible metric $d$ on $X$ and $n \in \bN$, we obtain another compatible metric $d_n$ on $X$ defined as
$$d_n(x, x'):=\max_{0 \leq i \leq n-1} d(\varphi^i(x), \varphi^i(x')).$$
It is easy to check that the sequence $a_n:= \Wdim_\varepsilon(X, d_n)$ is {\it subadditive}  in the sense that $a_{n+m}\leq a_n+a_m$ for any $n, m  \in \bN$. Thus we obtain a well-defined quantity:
$$\mdim_\varepsilon(X, d):=\lim_{n \to \infty} \frac{\Wdim_\varepsilon(X, d_n)}{n}.$$

\begin{definition}
The {\it mean (topological) dimension} of $(X, \varphi)$ is
$$\mdim(X, \varphi):=\sup_{\varepsilon > 0} \mdim_\varepsilon(X, d).$$
\end{definition}
By compactness of $X$ it follows that the definition is independent of choices of compatible metrics.
\begin{remark}
Lindenstrauss and Weiss introduced another equivalent definition of mean (topological) dimension in terms of the overlapping number of finite open covers \cite{LW00}. This definition works for actions on general compact Hausdorff spaces. See a friendly introduction to mean dimension theory in \cite{C15}.
\end{remark}

\subsection{Inverse limit}
Let $\{X_n\}_{n=1}^\infty$ be a sequence of compact metrizable spaces with surjective continuous maps $\psi_n \colon X_{n+1} \to X_n$ for all $n \in \bN$. The {\it inverse limit of $\psi_n$'s }, denoted as
$\varprojlim X_n,$
is defined as the subset of $\prod_{n \geq 1} X_n$ consisting of all elements $(x_n)_{n\in \mathbb{N}}$ satisfying $\psi_n(x_{n+1})=x_n$ for all $n \in \bN$.

Now let $\varphi_n \colon X_n \to X_n$ be a surjective continuous map  such that $\psi_n \circ \varphi_{n+1}=\varphi_n \circ \psi_n$ for every $n \in \bN$. In other words, $\psi_n \colon (X_{n+1}, \varphi_{n+1}) \to (X_n, \varphi_n)$'s is a sequence of factor maps of dynamical systems $(X_n, \varphi_n)$'s. Define $\Phi \colon  \varprojlim X_n \to \varprojlim X_n$ by sending $(x_n)_{n\in \mathbb{N}}$ to $(\varphi_n(x_n))_{n\in \mathbb{N}}$.

\begin{definition}
The dynamical system  $(\varprojlim X_n, \Phi)$ is called the {\it inverse limit of dynamical systems} $(X_n, \varphi_n)$'s.
\end{definition}

\begin{definition}\label{natural extension}
	For a dynamical system $(X, \varphi)$, setting $X_n =\cap_{k \geq 1} \varphi^k(X)$ and $\psi_n =\varphi_n =\varphi$ for all $n \in \bN$, the induced inverse limit is called the {\it natural extension} of $(X, \varphi)$, which we denote by $(\widetilde{X}, \widetilde{\varphi})$.
\end{definition}

The following proposition gives an upper bound of mean dimension for an inverse limit of dynamical systems \cite[Proposition 5.8]{Shi21}.
\begin{proposition}  \label{inverse decrease}
For an inverse limit of dynamical systems $(X_n, \varphi_n)$'s, we have
$$\mdim(\varprojlim X_n, \varprojlim \varphi_n) \leq \liminf_{n \to \infty} \mdim(X_n, \varphi_n).$$
In particular, for a dynamical system $(X, \varphi)$, we have
$$\mdim(\widetilde{X}, \widetilde{\varphi}) \leq \mdim(X, \varphi).$$
\end{proposition}

\subsection{Cellular automata}
Let us recall the definition of cellular automata \cite{CC10B, LM95B}.

\begin{definition} \label{CA}
For a compact metric space $X$ and a continuous map $f \colon X^I \to X$ with  a finite set $I \subseteq \bZ$,  the {\it cellular automaton} on $X^\bZ$ with local rule $f$ is the continuous map $F \colon X^\bZ \to X^\bZ$ sending $(x_n)_{n \in \bZ}$ to $(f((x_{n+j})_{j \in I}))_n$. We say that $F$ is {\it algebraic} if $X$ is a compact metric abelian group and $f$ is a group homomorphism. We call that $F$ is a {\it unit cellular automaton} if $I=\{1\}$.

\end{definition}

For the computation of mean dimension of cellular automata, we refer to \cite{BS21}.

\section{Mean rank and mean dimension}

In this section,  we consider the notion of mean rank for an endomorphism on a discrete abelian group,  discuss some useful properties, and establish an equality relating mean rank with mean dimension.

For a discrete abelian group $A$ denote by $\End(A)$ the ring of all endomorphisms on $A$. Fix $\varphi \in \End(A)$. For each finite subset $E$ of $A$, consider the subgroup of $A$ generated by $\{\varphi^i(E): i=0, \cdots, n-1\}$, which we denote by $T_n(E, \varphi)$. We call $T_n(E, \varphi)$ the {\it $n$-trajectory of $E$}. For a subgroup $B$ of $A$, we denote by $\rk(B)$ the rank of $B$. Since the sequence $\{\rk(T_n(E, \varphi))\}_{n\in \mathbb{N}}$ is subadditive, the limit $\lim_{n \to \infty} \frac{\rk(T_n(E, \varphi))}{n}$ exists, which we denote by $\mrk(E, \varphi)$.

\begin{definition}
	For  a discrete abelian group $A$ and $\varphi \colon A \to A$ a group homomorphism,
the {\it mean rank} of $A$ is defined as
$$\mrk(A, \varphi):=\sup_E \mrk(E, \varphi)$$
for $E$ ranging over all finite nonempty subsets of $A$.
\end{definition}

The following proposition whose proof is trivial shows that mean rank is upper continuous.
\begin{proposition} \label{upper continuous}
Let $\varphi \in \End(A)$ on a discrete abelian group $A$ and $\{A_n\}_{n \in \bN}$ an increasing sequence of subgroups of $A$ such that $\varphi (A_n) \subseteq A_n$ with the union $\cup_{n \in \bN} A_n=A$. Then $\mrk(A, \varphi)=\sup_{n \in \bN} \mrk(A_n, \varphi)$.
\end{proposition}

Let $\{A_n\}_{n=1}^\infty$ be a sequence of discrete abelian groups  $\varphi_n \in \End(A_n)$ for every $n \in \bN$.  Suppose that $\{\psi_{n, m} \colon A_n \to A_m \}$ is a directed system of  discrete abelian groups $A_n$'s such that $
\psi_{n,m} \circ \varphi_n =\varphi_m \circ \psi_{n, m}$ for all $n,m \in \bN$ with $n \leq m$.
Denote by $\lambda_n$ the natural embedding $A_n \to \oplus_{n \geq 1}A_n$.  Then an explicit construction of the colimit of $\{\psi_{n,m} \colon A_n \to A_m \}$
is
$$\varinjlim A_n:=(\oplus_{n \geq 1}A_n)/S$$
where $S$ is the subgroup of $\oplus_{n \geq 1}A$ generated by all elements $$\lambda_m \circ \psi_{n,m}(a_n)-\lambda_n(a_n)$$
with $n \leq m$ and $a_n \in A_n$ (see \cite[Proposition B-7.7]{Rotman15}).

The compatible $\varphi_n$'s then induce a homomorphism $\Phi \colon \varinjlim A_n \to \varinjlim A_n$
sending $\overline{(a_n)_n}$ to $\overline{(\varphi_n(a_n))_n}$.
In particular, if $A_n=A$ and $\psi_{n,m}=\varphi_n=\varphi$ for all $n, m \in \bN$, we denote by $\raA$ the colimit and $\overrightarrow{\varphi}$  the corresponding homomorphism on $\raA$.

\begin{remark}
Denote by $\bZ[x]$ the polynomial ring with integral coefficients and $\bZ[x^\pm]$ the Laurent polynomial ring correspondingly.  Fix $\varphi \in \End(A)$ on a discrete abelian group $A$. We can view $A$ as a right $\bZ[x]$-module via $a.x^n:=\varphi^n(a)$ for every $a \in A$ and $n \in \bN$. Define $$\varphi\otimes{\rm id} \colon A\otimes_{\bZ[x]} \bZ[x^\pm] \to A\otimes_{\bZ[x]} \bZ[x^\pm]$$ by sending $a\otimes x^k$ to $\varphi(a)\otimes x^k$. Then the map
$$\pi \colon (\raA, \overrightarrow{\varphi}) \to (A\otimes_{\bZ[x]} \bZ[x^\pm], \varphi\otimes {\rm id})$$
sending $\overline{(a_n)_n}$ to $\sum_{n=1}^\infty a_n\otimes x^{-n}$ is a group isomorphism such that $\pi \circ \overrightarrow{\varphi} = (\varphi \otimes {\rm id}) \circ \pi$. Under this isomorphism, we may view $(\raA, \overrightarrow{\varphi})$ as a localization of $(A, \varphi)$.
\end{remark}

Adapting \cite[Proposition 3.5]{CT15} into our situation, we obtain the following.

\begin{proposition} \label{mrk reduction}
 Let $\{\psi_{n, m} \colon A_n \to A_m \}$ be a  directed system of discrete abelian groups  and $\varphi_n \in \End(A_n)$ such that $\psi_{n,m} \circ \varphi_n =\varphi_m \circ \psi_{n, m}$ for all $n,m \in \bN$ with $n \leq m$. Denote by $A_n'$ the quotient group $A_n/\cup_{m \geq n}\ker (A_n \stackrel{\psi_{n, m}}{\to} A_m)$ and $\varphi_n'$ the map $A_n' \to A_n'$ sending $\overline{x}$ to $\overline{\varphi_n(x)}$. Then we have
$$\mrk(\varinjlim A_n, \Phi)=\sup_{n \geq 1} \mrk(A_n', \varphi_n').$$
In particular, if $\varphi \in \End(A)$ on a discrete abelian group $A$ is injective, then $\mrk(A, \varphi)=\mrk(\overrightarrow{A}, \overrightarrow{\varphi})$.
\end{proposition}

\begin{proof}
Consider the injective homomorphism $\iota_n \colon A_n' \to \varinjlim A_n$ by sending $\overline{x}$ to $\overline{\lambda_n(x)}$. We have that $\iota_n\circ \varphi_n=\Phi\circ \iota_n$  and $\iota_n(A_n')$ increases with $\cup_{n \geq 1} \iota_n(A_n')= \varinjlim A_n$. By Proposition \ref{upper continuous}, we have that $\mrk(A_n', \varphi_n')$ increases to $\mrk(\varinjlim A_n, \Phi)$.

On the other hand, for a single injective homomorphism $\varphi \colon A \to A$, we have $A_n'=A$ and $\varphi_n'=\varphi$. Thus
$$\mrk(\overrightarrow{A}, \overrightarrow{\varphi})=\sup_{n \geq 1} \mrk(A, \varphi) =\mrk(A, \varphi).$$
\end{proof}

Now we consider a compact metrizable abelian group $X$. The {\it Pontryagin dual} of $X$, denoted as $\hX$, is defined as the collection of continuous group homomorphisms from $X$ to the unit circle $\bT$. Under the pointwise multiplication and compact-open topology $\hX$ is a discrete abelian group.  The classical Pontryagin duality says that the Pontryagin dual of $\hX$ is isomorphic to $X$ as topological groups \cite{HR63}.

Let $\varphi \colon X \to X$ be a continuous group homomorphism. Then $\varphi$ induces a group homomorphsim $\widehat{\varphi} \colon \hX \to \hX$ sending $\chi$ to $\chi\circ \varphi$. For convenience, we may write
$$\langle \widehat{\varphi}(\chi), x\rangle=\langle \chi, \varphi(x)\rangle=\chi(\varphi(x))$$
for every $\chi \in \hX$ and $x \in X$.

The following theorem establishes a connection between mean dimension and mean rank for a general endomorphism (not necessarily invertible). The proof for an automorphism is a special case of \cite[Theorem 4.1]{LL18}.  We adapt the proof of \cite[Theorem 4.1]{LL18} into our setting.
\begin{theorem} \label{dual equality}
Let $X$ be a compact metrizable group and $\varphi \colon X \to X$  a  continuous endomorphism. Then $\mdim(X, \varphi)=\mrk(\widehat{X}, \widehat{\varphi})$.
\end{theorem}

\begin{proof}
 Let us first estimate the upper bound of mean dimension. Fix a compatible metric $d$ on $X$. We may consider $X$ as the Pontryagin dual of $\hX$. Then for any $\varepsilon > 0$, there exists a finite subset $A$ of $\hX$  such that whenever $x|_A=x'|_A$ for $x, x' \in X$ we have $d(x, x') < \varepsilon$.

 Fix $n \in \bN$ and consider a continuous map $\pi_n \colon X \to \bT^{A \times n}$ by defining
 $$(\pi_n(x))(a, i)=<a, \varphi^i(x)>$$
 for every $x \in X, a \in A$, and $i=0, 1, \cdots, n-1$. Clearly by the choice of $A$ we have that
 $\pi_n$ is an  $(\varepsilon, d_n)$-embedding and hence
$$\Wdim_\varepsilon(X, d_n) \leq \dim (\Ima(\pi_n)).$$

To compute $\dim (\Ima(\pi_n))$, by a Pontryagin's result \cite[Theorem 24.28]{HR63}, we first have  $\dim (\Ima(\pi_n))=\rk(\widehat{\Ima(\pi_n)})$.
Consider the surjective homomorphism $\psi_n \colon X \to \Ima(\pi_n)$ sending $x$ to $\pi_n (x)$. By Pontryagin duality, the map $\widehat{\psi_n} \colon \widehat{\Ima(\pi_n)} \to \hX$ is injective and hence $$\rk(\widehat{\Ima(\pi_n)})=\rk(\Ima(\widehat{\psi_n}))=\rk(\Ima(\widehat{\pi_n})).$$
By the definition of $\pi_n$, we have that $\widehat{\pi_n}$ sends $(\lambda_{a, i})_{a \in A, 0\leq i < n}$ to $\sum_{a \in A, 0\leq i < n} \lambda_{a,i} \widehat{\varphi}^i(a)$, where $\lambda_{a, i} \in \bZ=\widehat{\bT}$ for all $a \in A$ and $0 \leq i < n$. Thus  $\Ima(\widehat{\pi_n})=T_n(A, \widehat{\varphi})$. Therefore,
$$\Wdim_\varepsilon(X, d_n) \leq \dim (\Ima(\pi_n))=\rk(T_n(A, \widehat{\varphi})).$$
Letting $n \to \infty$ and then  $\varepsilon \to 0$, we obtain $\mdim(X, \varphi) \leq \mrk(\hX, \widehat{\varphi})$.

Now we estimate the lower bound of mean dimension. Fix a finite subset $A \subseteq \hX$ and  $n \in \bN$. Let $\cE \subseteq \cup_{i=0}^{n-1}\widehat{\varphi}^i(A)$ be a maximal linearly independent subset and hence $|\cE|=\rk(T_n(A, \widehat{\varphi}))$. Then it suffices to show that for every $0 < \varepsilon < 1/2$ and $n \in \bN$, one has
$$\rk(T_n(A, \widehat{\varphi}))\leq \Wdim_\varepsilon(X, d_n).$$
Denote by $\langle\cE\rangle$ the abelian subgroup generated by $\cE$. By Zorn's Lemma there exists a maximal subgroup $H$ of $\hX$ subject to the condition that $H \cap \langle\cE\rangle=0$. Then for any $u \in \hX$ there exists $k \in \bZ$ such that $ku \in H \oplus \langle\cE\rangle$ (Otherwise we will get a strictly larger subgroup $H\oplus \bZ u$). Write $ku$ as
$$ku=\sum_{e \in \cE} \lambda_e e +v$$
with $v \in H$ and $\lambda_e \in \bZ$.
Based on such a decomposition and identifying $X$ with the Pontryagin dual of $\hX$, we can define a continuous map $f \colon ([0, 1/2]^\cE, ||\cdot||_\infty) \to X$ sending $x=(x_e)_{e \in \cE}$ to $f(x)$, where for each $u \in \hX$,  the image $f(x)(u)$ is defined as
$$f(x)(u)=\sum_{e \in \cE} \frac{\lambda_e}{k}x_e +\bZ \in \bT=\bR/\bZ.$$
   It is straightforward to check that $f$ is well-defined, continuous, and injective.

Now we  introduce a compatible metric $d$ on $X$ such that for every $x,x' \in [0, 1/2]^\cE$ we have
$$||x-x'||_\infty \leq d_n(f(x), f(x')).$$
Therefore, by Lemma \ref{Wdim base}, for every $0 < \varepsilon < 1/2$ we have
$$\rk(T_n(A, \widehat{\varphi}))=|\cE|=\Wdim_\varepsilon([0, 1/2]^\cE, ||\cdot||_\infty) \leq \Wdim_\varepsilon(X, d_n).$$
Enumerate the elements of $\hX$ as $a_1, a_2, \cdots, $ such that $A=\{a_1, \cdots, a_{|A|}\}$. Let $\vartheta$ be the compatible metric on $\bT$ given by
$$\vartheta(x+n, y+m)=\min_{k \in \bZ} |x-y+k|.$$
Then the desired metric $d$  on $X$ can be defined as
$$d(\chi, \chi'):=\max_{a \in A} \vartheta(\chi(a), \chi'(a))+\max_{j > |A|} \frac{1}{2^j}\vartheta(\chi(a_j), \chi'(a_j)).$$
We then complete the proof.
\end{proof}

\begin{remark}
In light of \cite[Definition 4.2]{DFG20}, we can extend the notion of mean rank to the actions by amenable semigroups and expect the above connection holds.
\end{remark}

\section{mean dimension of the natural extension}

In this section we give the proof of Theorem \ref{main thm} and Corollary \ref{CA application}. For a dynamical system $(X, \varphi)$, recall that a point $x \in X$ is called {\it non-wandering} if for every open set $U$ containing $x$ and $N \geq 1$, there exists  $k \geq N$ such that $U \cap \varphi^{-k}(U) \neq \emptyset$. Denote by $\Omega(X)$ the collection of non-wandering points of $(X, \varphi)$.
\begin{lemma} \label{nonsurjective reduction}
For a dynamical system $(X, \varphi)$, we have $\mdim(X, \varphi)=\mdim(\cap_{n \geq 1} \varphi^n(X), \varphi)$.
\end{lemma}

\begin{proof}
Following the same argument of \cite[Lemma 7.2]{Gutman17} for the noninvertible case, we still have $\mdim(X, \varphi)=\mdim(\Omega(X), \varphi)$. Since $ X \supseteq \cap_{n \geq 1} \varphi^n(X) \supseteq \Omega(X)$, it follows that
$$\mdim(X, \varphi)\geq \mdim(\cap_{n \geq 1} \varphi^n(X), \varphi) \geq \mdim(\Omega(X), \varphi)=\mdim(X, \varphi).$$
\end{proof}

Now we are ready to prove Theorem \ref{main thm}.

\begin{proof}[Proof of Theorem \ref{main thm}]

First we assume that $\varphi$ is surjective. From Proposition \ref{inverse decrease}, it suffices to show
$\mdim(X, \varphi) \leq \mdim(\tX, \tphi)$.

Consider a map $\pi \colon \overrightarrow{\hX} \to \widehat{\widetilde{X}}$ defined in the following way. For  each element $\overline{(\chi_n)_{n \in \bN}} \in \overrightarrow{\hX}$, $\pi(\overline{(\chi_n)_{n \in \bN}})$ maps each $(x_n)_{n \in \bN} \in \tX$ to $\sum_{n \in \bN}\chi_n(x_n)$. By definition, it is straightforward to check that $\pi$ is well-defined and injective.

 Note that
$\overrightarrow{\widehat{\varphi}}$ sends $\overline{(\chi_n)_{n\in \mathbb{N}}}$ to $\overline{(\chi_n\circ \varphi)_{n\in \mathbb{N}}}$ and for each $\xi \in \widehat{\widetilde{X}}$, $\widehat{\tphi}(\xi)$ sends each $(x_n)_{n\in \mathbb{N}} \in \widetilde{X}$ to $\xi((\varphi(x_n)_{n\in \mathbb{N}}))$. Thus $\pi \colon \overrightarrow{\hX} \to \widehat{\widetilde{X}}$ is an injective group homomorphism such that $\pi \circ \overrightarrow{\widehat{\varphi}} =\widehat{\tphi} \circ \pi$. It follows that  $\mrk(\overrightarrow{\hX}, \overrightarrow{\hphi}) \leq \mrk(\widehat{\widetilde{X}}, \widehat{\tphi})$.

 Since $\varphi$ is surjective, by Pontryagin duality,  $\widehat{\varphi}$ is injective. By Proposition \ref{mrk reduction}, we obtain $\mrk(\hX, \hphi) =\mrk(\overrightarrow{\hX}, \overrightarrow{\hphi})$.
 Applying Theorem \ref{dual equality} twice we have
 \begin{align*}
     \mdim(X, \varphi) &=\mrk(\hX, \hphi) =\mrk(\overrightarrow{\hX}, \overrightarrow{\hphi}) \leq \mrk(\widehat{\widetilde{X}}, \widehat{\tphi})=\mdim(\tX, \tphi).
 \end{align*}
This completes the proof in the case when $\varphi$ is surjective.

 Now we assume that $\varphi$ is any continuous endomorphism.
 Denote by $Y$ the subspace $\cap_{n \geq 1} \varphi^n(X)$. Since the restriction of $\varphi$ on $Y$ is surjective, we have $\mdim(\widetilde{Y}, \widetilde{\varphi})=\mdim(Y, \varphi)$. By Lemma \ref{nonsurjective reduction} and the definition of $(\widetilde{X}, \widetilde{\varphi})$, we obtain that
 $$\mdim(X, \varphi)=\mdim(Y, \varphi)=\mdim(\widetilde{Y}, \tphi)=\mdim(\tX, \tphi).$$
\end{proof}

Next, we prove Corollary \ref{CA application}.
\begin{proof}[Proof of  Corollary \ref{CA application}.]
Since $f$ is a continuous homomorphism between compact metric abelian groups, by \cite[Lemma 7.1]{BS21},  $F$ is also a continuous homomorphism. Applying Theorem \ref{main thm} to the induced dynamical system $(X^\bZ, F)$, we have  $\mdim(\widetilde{X^\bZ}, \widetilde{F})=\mdim(X^\bZ, F)$.
\end{proof}

Finally, by a modification of the proof of Theorem \ref{main thm}, we obtain an inverse limit version of Theorem \ref{main thm} which is interesting in its generality.

\begin{theorem}
Let  $(\varprojlim X_n, \Phi)$ be an inverse limit of dynamical systems $(X_n, \varphi_n)$'s such that $X_n$ is a compact metrizable abelian group and $\varphi_n, \psi_n$ are group homorphisms for every $n \in \bN$.  Then we have
$$\mdim(\varprojlim X_n, \Phi)=\lim_{n \to \infty} \mdim(X_n, \varphi_n)=\sup_{n \in \bN} \mdim(X_n, \varphi_n).$$
\end{theorem}

\begin{proof}
Since every $\psi_n$ is surjective, by Pontryagin duality, $\widehat{\psi_n}$ is injective. Denote by $\Psi$ the corresponding map on the colimit $\varinjlim \widehat{X_n}$. By Proposition \ref{mrk reduction}, we have
$$\mrk(\varinjlim \widehat{X_n}, \Psi) =\lim_{n \to \infty}\mrk(\widehat{X_n}, \widehat{\varphi_n})=\sup_{n \in \bN} \mrk(\widehat{X_n}, \widehat{\varphi_n}).$$
By Proposition \ref{inverse decrease}, it suffices to prove that $\sup_{n \in \bN}\mdim(X_n, \varphi_n) \leq \mdim(\varprojlim X_n, \Phi)$.
Applying Theorem \ref{dual equality} twice, we obtain

\begin{align*}
&\sup_{n \in \bN} \mdim(X_n, \varphi_n) \\
&=\sup_{ n \in \bN} \mrk(\widehat{X_n}, \widehat{\varphi_n}) \\
&=\mrk(\varinjlim \widehat{X_n}, \Psi) \\
&\leq \mrk(\widehat{\varprojlim X_n}, \widehat{\Phi}) =\mdim(\varprojlim X_n, \Phi).
\end{align*}

\end{proof}

\medskip
\noindent{\it Acknowledgments.}  We are grateful for the careful reading of the anonymous referee. B. L. is supported by NSFC grant 12271387. R.S
was partially supported by Fondation Sciences Math\'{e}matiques de Paris.


  \end{document}